\documentclass{article}
\usepackage[utf8]{inputenc}

\title{Maximizing directed cycles in tournaments}

\usepackage{amsmath, amssymb, amsthm, url, tikz, verbatim, tkz-graph}
\usepackage{blkarray}
\usetikzlibrary{arrows.meta}

\newtheorem{theorem}{Theorem}[section]
\newtheorem{proposition}[theorem]{Proposition}
\newtheorem{lemma}[theorem]{Lemma}

\newtheorem{corollary}[theorem]{Corollary}

\newtheorem{question}[theorem]{Question}

\newtheorem{conjecture}[theorem]{Conjecture}

\DeclareMathOperator{\Tr}{Tr}
\DeclareMathOperator{\diag}{diag}
\usetikzlibrary{calc}
\usetikzlibrary{patterns}
\usetikzlibrary{decorations.markings}
\usetikzlibrary{arrows,shapes.geometric,%
	decorations.pathreplacing,shapes,shadows}

\date{}

\author{ Yijia Fang\thanks{Department of Mathematics, National University of Singapore. Email: fangyijia@u.nus.edu.}
\and
Hao Huang \thanks{Department of Mathematics, National University of Singapore. Email: huanghao@nus.edu.sg. Research supported in part by a start-up grant at NUS and an MOE Academic Research Fund (AcRF) Tier 1 grant A-8003627.}}

\newcommand{\floor}[1]{\left\lfloor #1 \right\rfloor}

\begin{document}

\maketitle

\begin{abstract}
    Determining the combinatorial structures that maximize the number of prescribed substructures is a central theme in extremal combinatorics. Grzesik, Kr\'al', Lov\'asz and Volec showed that when $\ell$ is not divisible by $4$, the random tournament contains asymptotically the most directed cycles of length $\ell$ among all $n$-vertex tournaments. In the paper, we resolve the remaining cases where $\ell$ is divisible by $4$. We show that, in this regime, the so-called carousel tournament asymptotically maximizes the number of directed $\ell$-cycles among all $n$-vertex tournaments, and in particular contains strictly more such cycles than the random tournament. This confirms the conjecture of Bartley and Day.
\end{abstract}

\section{Introduction}


Estimating the number of copies of a graph in a family of graphs is a fundamental problem in extremal graph theory. A major open problem in this area is \textit{Sidorenko's conjecture}, posed by Erd\H os and Simonovits \cite{Simonovits1984} and Sidorenko \cite{Sidorenko1993} independently. Roughly speaking, Sidorenko’s conjecture states that, the homormophism density of any fixed bipartite graph $H$ is minimized by a random graph, among graphs with a prescribed edge density.
\begin{conjecture}[Sidorenko]
    For every bipartite graph $H$ with $v$ vertices and $e$ edges, and graph $G$ on n vertices with edge density $p>0$, there are at least $(p^e+ o(1))n^v$ copies of H in G.
\end{conjecture}

While Sidorenko's conjecture is a natural characterization of extremal properties a random graph might have, the corresponding problem for directed graphs, or digraph for short is considerably more intricate. Many such extremal results in the directed setting focus on finding the extremal number of not necessarily induced copies of a fixed digraph $H$ in a particular family of directed graphs, tournaments. A tournament is an orientation of a complete graph. The random tournament is obtained by orienting every edge of the complete graph independently in either direction with probability $\frac{1}{2}$. Even in this restricted setting, a digraph $H$ may be either Sidorenko, meaning that its density is minimized by the random tournament; anti-Sidorenko, meaning that its density is maximized by the random tournament; or neither. An elegant result by Zhao and Zhou \cite{ZhaoZhou2020} completely characterizes impartial digraphs, namely, those for which the number of copies in a host tournament depends only on the order of the tournament. For directed paths $P_k$ of fixed length $k$, Sah, Sawhney, and Zhao \cite{SahSawhneyZhao2023} proved that among all $n$-vertex tournaments, the random tournament (and actually any regular tournament) contains asymptotically the maximum number of copies of $P_k$. These results confirmed unpublished conjectures of Fox, Huang, and Lee. A substantial body of work has investigated analogous extremal questions for other fixed digraphs $H$. Notable examples include the work of
Coregliano and Razborov \cite{CR2017} on transitive subtournaments $H$, Burke, Lidick\'y, Pfender and Philips \cite{BLPP2021} on all tournaments with at most $4$ vertices,  Coregliano, Parente, and Sato \cite{CPS2019} on certain five-vertex tournaments. More recently, the two papers by Fox, Himwich, Mani, and Zhou \cite{FHMZ2024, FHMZ2025}, the work of He, Mani, Nie, Tung, and Wei \cite{HMNTW2025}, and the paper of Chen, Clemen and Noel \cite{CCN2026} all aim to characterize additional anti-Sidorenko oriented graphs and explore their connections with the corresponding undirected extremal problems.

Unlike the paths and impartial digraphs discussed earlier, directed cycles do not always attain their maximum count in the random tournament. In fact, their extremal behavior highly depends on the cycle length $\ell$. Let $C(n,\ell)$ be the maximum number of directed cycles of length $\ell$ in an $n$-vertex tournament, and $R(n,\ell)=\frac{(\ell-1)!}{2^{\ell}}\binom{n}{\ell}$ be the expected number of directed cycles of length $\ell$ in the random $n$-vertex tournament. Define the limit of their ratio
\[
    c(\ell)=\lim_{n\to\infty}\frac{C(n,\ell)}{R(n,\ell)}.
\]
An early result of Kendall and Smith \cite{KS1940} in 1940 implies that $c(3)=1$, attained by the random tournament. Two decades later, Colombo \cite{Colombo1964}, and Beineke and Harary \cite{Beineke1965} established $c(4)=4/3$, with the maximum attained by the so-called carousel tournament. Here a \textit{carousel tournament} $\mathrm{Car}_n$ with $n$ vertices is defined as the directed Cayley graph \(\textrm{Cay}(\mathbb{Z}/n\mathbb{Z}, \{1,\cdots, \frac{n-1}{2}\})\) for odd $n$. In other words, the vertices are $\mathbb{Z}/n\mathbb{Z}=\{\overline{0}, \overline{1}, \cdots, \overline{n-1}\}$, and there is a directed edge from $i$ to $j$ if $j-i \in \{\overline{1}, \cdots, \overline{\frac{n-1}{2}}\}$. For even $n$, $\mathrm{Car}_n$ can be obtained from $\mathrm{Car}_{n+1}$ by deleting one vertex. More recently in 2016, Savchenko \cite{Savchenko2016} established
bounds on directed $5$-cycles and $6$-cycles in regular tournaments. And in 2017, Komarov and Mackey \cite{KM2017} proved that $c(5)=1$, with the maximum again given by the random tournament. 

\begin{figure}[ht]
\centering
\begin{tikzpicture}[
    scale=0.7,
    transform shape,
    vertex/.style={
        circle,
        draw,
        fill=white,
        minimum size=7mm,
        inner sep=0pt
    },
    edge/.style={
        -{Stealth[length=2mm]},
        thin,
        shorten <=3.5mm,
        shorten >=3.5mm
    }
]

\foreach \i in {0,...,6} {
    \coordinate (p\i) at ({90-360*\i/7}:3cm);
}

\foreach \i in {0,...,6} {
    \foreach \d in {1,2,3} {
        \pgfmathtruncatemacro{\j}{mod(\i+\d,7)}
        \draw[edge] (p\i) -- (p\j);
    }
}

\foreach \i in {0,...,6} {
    \node[vertex] at (p\i) {$\i$};
}

\end{tikzpicture}
\caption{The $7$-vertex carousel tournament $\textup{Car}_7$}
    \label{fig:cycle}
\end{figure}
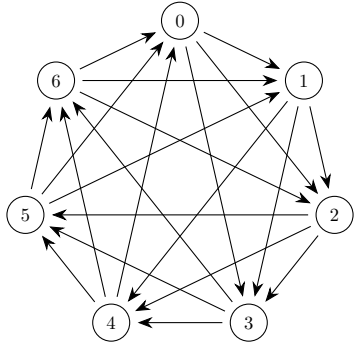

Based on these results, Bartley \cite{Bartley2018} and Day \cite{Day2017} conjectured that when $\ell$ is not divisible by $4$, the random tournament is asymptotically optimal, while for $\ell$ divisible by $4$, the asymptotic maximizer is the carousel tournament. The first part of their conjecture was confirmed in 2023 by Grzesik, Kr\'al', Lov\'asz, and Volec \cite{Grzesik2023}. In other words, they proved that $c(\ell)=1$ when $\ell$ is not divisible by $4$. For the remaining cases where $\ell=4k$, by calculations, Bartley \cite{Bartley2018} formulated this conjecture in the following form:
\[
    c(4k) = 1+2 \cdot \sum_{i=1}^{\infty} \left(\frac{2}{(2i-1)\pi}\right)^{4k}.
\]
Grzesik et al. \cite{Grzesik2023} proved $c(4k)\le 1+(\frac{2}{\pi}+o(1))^{4k}$, matching the order of the conjectured value of $c(4k)$ when $k \rightarrow \infty$. They also showed that $c(8)=\frac{332}{315}$ as conjectured, leaving the exact value of $c(4k)$ for $k\ge 3$ unknown.


In this paper, we resolve the remaining cases $n=4k$ for every positive integer $k$, confirming the conjecture by Bartley \cite{Bartley2018} and Day \cite{Day2017}.

\begin{theorem} \label{thm_main}
Let \(k\) be a fixed positive integer, then the
maximum number of directed cycles of length \(4k\) in an \(n\)-vertex tournament is
equal to
\[\left(1+2 \cdot \sum_{i=1}^{\infty} \left(\frac{2}{(2i-1)\pi}\right)^{4k}+o(1)\right)\frac{(4k-1)!}{2^{4k}}\binom{n}{4k}\]
when $n$ goes to infinity. In other words, $c(4k)=1+2 \cdot \sum_{i=1}^{\infty} \left(\frac{2}{(2i-1)\pi}\right)^{4k}$. Furthermore, this maximum can be asymptotically attained by the carousel tournament.
\end{theorem}

It is not hard to observe that as $n \rightarrow \infty$, counting directed cycles of fixed length becomes asymptotically equivalent to counting closed directed walks of the same length. So, the problem boils down to estimating the trace of the powers of the tournament's adjacency matrix. We apply techniques from linear algebra to establish two results on the trace of the powers of skew-symmetric matrices, which together yield our main result, Theorem \ref{thm_main}. We believe these two theorems are of independent interest beyond their use here.

\begin{theorem} \label{thm_non_negative}
For a fixed integer \(k \ge 1\), and an arbitrary
positive integer \(n\), for every skew-symmetric \(n \times n\) matrix
\(A\) with off-diagonal entries in \(\{-1, 1\}\), we have
\[\Tr[(A+J)^{4k}] \le \Tr[A^{4k}]+\Tr[J^{4k}],\]where \(J\) is the all-one
\(n \times n\) matrix.
\end{theorem}

A \textit{signed permutation matrix} is a square matrix with exactly one nonzero entry equal to either $1$ or $-1$ in each row and each column. A signed permutation matrix $Q$ is always orthogonal, i.e. $Q^TQ=I$ where $I$ is the identity matrix, and hence $\Tr[Q^TAQ]=\Tr[A]$ for any matrix $A$.

\begin{theorem} \label{thm_trace}
For a fixed integer \(k \ge 1\), and an arbitrary
positive integer \(n\), for every skew-symmetric \(n \times n\) matrix
\(A\) with off-diagonal entries in \(\{-1, 1\}\), we have
\[\Tr[A^{4k}] \le \Tr[T_n^{4k}],\]where \(T_n\) is the skew-symmetric
matrix with \[
T_n=(t_{ij})_{i,j=1}^n,\qquad  
t_{ij}=  
\begin{cases}  
1,&i<j,\\
-1,&i>j,\\
0,&i=j.  
\end{cases}
\] 
Furthermore, the equality holds if and only if $A=Q^TT_nQ$ for some signed permutation matrix $Q$.
\end{theorem}

\section{Proof of Theorem \ref{thm_non_negative}}

\begin{lemma} \label{lem_lambda_mu}

Let \(K\) be a real skew-symmetric \(n \times n\)
matrix whose eigenvalues are $\{-i\mu_j\}_{1 \le j \le n}$ for real numbers $\mu_j$'s, and \(u\) be a real unit column vector of dimension \(n\), and 
\(P=uu^T\). Let \(M=K+P\) and denote by 
\(\lambda_1, \cdots, \lambda_n\) its (possibly non-real) eigenvalues. Then for every positive integer
\(p \ge 1\), we have \[\sum_{i=1}^n |\lambda_i|^p \le 1+\sum_{j=1}^n |\mu_j|^p.\]
\end{lemma}

\begin{proof} 
We write \(M\) using its real Schur decomposition,
namely \(T=Q^T M Q\) for some real orthogonal matrix \(Q\) and an upper
quasi-triangular matrix \(T\) such that
\[T=\begin{bmatrix}T_{11} & T_{12} & \cdots & T_{1m}\\
0 &T_{22} & \cdots & T_{2m}\\
\vdots & \ddots & \ddots & \vdots\\
0 & \cdots & 0 & T_{mm}\end{bmatrix},\]such that the diagonal blocks \(T_{ii}\) are either
\(1 \times 1\) or \(2 \times 2\) matrices, corresponding to the real
eigenvalues and the conjugate pairs of eigenvalues of $M$ respectively.

Now let \(v=Q^T u\),  which is still a unit vector since $v^Tv=u^T QQ^T u=u^Tu=1$. We also define \(L=Q^TKQ\). It is not hard to see that $L$ is skew-symmetric as \(K\) is, since $L^T=Q^TK^TQ=-Q^TKQ=-L$. Therefore the diagonal entries of
\(L\) are zeros. Furthermore,
\[T=Q^TMQ=Q^T(K+P)Q=Q^T K Q +(Q^T u)(Q^T u)^T=L+vv^T.\]Therefore if we
write \(v^T=(v_1^T, \cdots, v_m^T)\) according to the blocks in this real Schur
decomposition of \(M\), then we can write \(L_{ij}=T_{ij}-v_i v_j^T\) to form the block decomposition for $L$.
There are two cases: if \(T_{ii}\) is a \(1 \times 1\) block, then $L_{ii}=0$ implies that as numbers, \(T_{ii}=v_i^2\), which equals the real eigenvalue it
corresponds to. This implies
\[\sum_{\lambda_i \in \mathbb{R}} |\lambda_i|^p = \sum_{i: \dim(T_{ii})=1} v_i^{2p} \le \sum_{i=1}^m \|v_i\|^{2} = 1.\]
So it remains to prove
\(\sum_{\lambda_i \not\in \mathbb{R}} |\lambda_i|^p \le \sum_{j=1}^n |\mu_j|^p\).
Without loss of generality, we may assume that these non-real eigenvalues of \(M\) are conjugate pairs
\(\{\lambda_j, \overline{\lambda_j}\}\) for \(j=1, \cdots, t\),
corresponding to $2 \times 2$ blocks \(T_{jj}\) of \(T\). And the
corresponding orthonormal real vectors from $Q$ are \(x_j, y_j\). It is not hard check that if we define \(z_j=(x_j+iy_j)/\sqrt{2}\) and
\(w_j=(x_j-iy_j)/\sqrt{2},\) then \(\{z_j\} \cup \{w_j\}\) are orthonormal
in \(\mathbb{C}^n\).

Since \(T_{jj}\) is a \(2 \times 2\) block, then assuming
\(L_{jj}=\begin{bmatrix} 0&c_{j}\\-c_{j}&0 \end{bmatrix}\), and
\(v_j=\begin{bmatrix} a_j\\b_j\end{bmatrix}\) with \(a_j=x_j^T u\),
\(b_j=y_j^T u\), then
\(T_{jj}=\begin{bmatrix} a_j^2 & c_j+a_jb_j\\-c_j+a_jb_j & b_j^2\end{bmatrix}\).
Its determinant is equal to \(c_j^2\), therefore the two conjugate
eigenvalues it correspond to both have absolute value \(|c_j|\). 

Set \(H=iK\). Since $K$ is skew-symmetric, \(H\) becomes Hermitian. Using the skew-symmetry of \(K\), we also
have \(x_j^T Kx_j=y_j^T K y_j=0\), and thus
\[\langle Hz_j, z_j \rangle=\frac{1}{2} (x_j^T (iK)(iy_j)-iy_j^T(iK)x_j)=\frac{1}{2}(-x_j^T K y_j+y_j^T K x_j)=-c_j.\]\[\langle Hw_j, w_j \rangle=\frac{1}{2}(x_j^T K y_j-y_j^T K x_j)=c_j.\]We
know \(H=iK\) is Hermitian and we assume that the eigenvalues of $K$ are $\{-i\mu_j\}$, so the eigenvalues of $H$ are real numbers $\{\mu_j\}_{j=1, \cdots, n}$, we can choose an orthonormal basis
\(\{e_1, \cdots, e_n\}\) such that \(He_r=\mu_r e_r\) for
\(\mu_r \in \mathbb{R}\), and \(e_i \in \mathbb{C}^n\). Now for a unit
vector \(z=\sum_r \alpha_r e_r\), we have \(\sum_r |\alpha_r|^2=1\), and
\(\langle Hz, z\rangle=\sum_r \mu_r |\alpha_r|^2\).
Therefore,\[|\langle Hz, z\rangle|^p =\left|\sum_r \mu_r|\alpha_r|^2 \right|^p \le \sum_r |\mu_r|^p |\alpha_r|^2=\sum_r |\mu_r|^p |\langle z, e_r\rangle|^2.\]Applying
this inequality to \(z_j\) and \(w_j\). We have
\[2\sum_{j=1}^t |\lambda_j|^p = \sum_{j=1}^t (|\langle Hz_j, z_j\rangle|^p+|\langle Hw_j, w_j\rangle|^p) \le \sum_r |\mu_r|^p \sum_{j=1}^t (|\langle z_j, e_r\rangle|^2+|\langle w_j, e_r\rangle|^2).\]Since
\(\{z_j\} \cup \{w_j\}\) are orthonormal, completing them to an
orthonormal basis gives
\[\sum_{j=1}^t \bigg(|\langle z_j, e_r\rangle|^2 + |\langle w_j, e_r\rangle|^2\bigg) \le |e_r|^2=1.\]Therefore
\(2\sum_{j=1}^t |\lambda_j|^p \le \sum_{r=1}^n |\mu_r|^p\). This immediately gives
\(\sum_{\lambda_j \not\in \mathbb{R}}|\lambda_j|^p \le \sum_{j=1}^n |\mu_j|^p\),
completing the proof of the lemma.
\end{proof}

Now we are ready to prove the main result in this section.\\

\noindent \textbf{Proof of Theorem \ref{thm_non_negative}:}\smallskip
\begin{proof} We take the unit vector
\(u=\frac{1}{\sqrt{n}} \vec{1}\), then \(P=uu^T=\frac{1}{n} J\). Let
\(K=\frac{1}{n} A\) and \(M=K+P\). The (possibly non-real) eigenvalues
of \(M\) are denoted by \(\lambda_1, \cdots, \lambda_n \in \mathbb{C}\). The eigenvalues of $K$ are denoted by $\{-i\mu_j\}$ with $\mu_j \in \mathbb{R}$. Then the eigenvalues of $A$ are $\{-in\mu_j\}$ and thus we have 
$$\Tr[A^{4k}]=n^{4k} \left( \sum_{j=1}^n (-i\mu_j)^{4k}\right)=n^{4k}\left( \sum_{j=1}^n \mu_j^{4k}\right).$$
Setting
\(p=4k\) in Lemma \ref{lem_lambda_mu}, we have
\[\sum_{i=1}^n |\lambda_i|^{4k} \le 1+\sum_{j=1}^n |\mu_j|^{4k}=1+\frac{\Tr[A^{4k}]}{n^{4k}}.\]
On the other hand, since for conjugate pair
\(\lambda, \overline{\lambda}\), we have
\(\lambda^{4k}+\overline{\lambda}^{4k} \le 2|\lambda|^{4k}\), therefore
\[\frac{\Tr[(A+J)^{4k}]}{n^{4k}}=\Tr[M^{4k}]=\sum_{i=1}^n \lambda_i^{4k} \le \sum_{i=1}^n |\lambda_i|^{4k} \le 1+\frac{\Tr[A^{4k}]}{n^{4k}}.\]
Noting that \(\Tr[J^{4k}]=n^{4k}\). This completes the proof.
\end{proof}

\section{Maximizing the trace of the matrix power}

We start by introducing some notation to be used throughout this section. For \(x\in\mathbb R^n\), let
\[
x^*=(x_1^*,\dots,x_n^*),\qquad  
\] be the non-increasing rearrangement of \((|x_1|,\dots,|x_n|)\).

For two non-negative non-increasing vectors \(u,v\), we write \(u\prec v\) or say $v$ weakly majorizes $u$ if
\(\sum_{j=1}^r u_j\le \sum_{j=1}^r v_j\) for every \(1\le r\le n\). A function $f$ is strictly convex if and only if for all $0<\alpha<1$ and $x_1<x_2$,
\[
    f(\alpha x_1+(1-\alpha)x_2)<\alpha f(x_1)+(1-\alpha)f(x_2).
\]
We have the following variant of Karamata's inequality by Tomic \cite{Tomic1949} and Weyl \cite{Weyl1949} about weakly majorization and strictly convex functions:
\begin{proposition}[{Karamata's inequality, see also
  \protect\cite[Chapter~3, Theorems~A.8.a and~C.1.a]{Marshall1979}}]\label{prop_karamata}
    If a function $f$ is strictly increasing and strictly convex on $[0,\infty)$, then for every pair of non-negative non-increasing vectors $u\prec v$, we have
    \[
        f(u_1)+\dots+f(u_n) \le f(v_1)+\dots+f(v_n),
    \]
    where the equality holds if and only if $u=v$.
\end{proposition}
In particular, $f(t)=t^2$ is strictly increasing and strictly convex on $[0,\infty)$, and hence $||u||_2^2\le||v||_2^2$ for every pair of non-negative non-increasing vectors $u\prec v$, where the equality holds if and only if $u=v$.

Define the symmetric matrix \(K=K_n\) by\\
\[
K_{ij}=  
\begin{cases}  
1,&i+j\le n,\\
0,&i+j=n+1,\\
(-1)^{i+j-n-2},&i+j\ge n+2.  
\end{cases}
\] For example when $n=6$,\\
\[
K_6=  
\begin{bmatrix}  
1&1&1&1&1&0\\
1&1&1&1&0&1\\
1&1&1&0&1&-1\\
1&1&0&1&-1&1\\
1&0&1&-1&1&-1\\
0&1&-1&1&-1&1  
\end{bmatrix}.
\] 
Later, we will see that this matrix is closely related to the matrix $T_n$ we defined earlier in Theorem \ref{thm_trace}. In fact, one can be obtained from the other by permuting the rows and columns, and flipping the signs of entire rows or columns. 

The first lemma in this section shows that as a linear operator, $K_n$ preserves weak majorization $\prec$.

\begin{lemma}\label{lem_monotonicity} 
If \(u \prec v\), then \(K_n u \prec K_n v\). Furthermore, if $K_nu=K_nv$, then $u=v$.
\end{lemma}
\begin{proof}
From the definition of weak majorization $u \prec v$, for every \(1 \le r \le n\), we have
\(\sum_{i=1}^r u_i \le \sum_{i=1}^r v_i\). Now for any \(r\),
\[\sum_{i=1}^r (K_nu)_i=\sum_{i=1}^r \sum_{j=1}^n (K_n)_{ij}u_j.\] 
We claim that the coefficient of \(u_j\), which is
\(\sum_{i=1}^r (K_n)_{ij}=:w_{j,r}\) is non-increasing in \(j\) for fixed $r$ and non-negative. We also let \(w_{n+1,r}=0\). Suppose true, then we have
\[\sum_{i=1}^r (K_n u)_i = \sum_{j=1}^n w_{j,r} u_j=\sum_{j=1}^n (w_{j,r}-w_{j+1,r})(\sum_{i=1}^j u_i) \le \sum_{j=1}^n (w_{j,r}-w_{j+1,r})(\sum_{i=1}^j v_i)= \sum_{j=1}^n w_{j,r} v_j = \sum_{i=1}^r (K_n v)_i.\]
Then the conclusion follows. 

To see the monotonicity and non-negativity of $w_{j,r}$, the calculation gives
\[w_{j,r}=
\begin{cases}
r,&r+j\le n,\\
n-j,&r+j\ge n+1\text{ and }r+j\not\equiv n\pmod 2,\\
n-j+1,&r+j\ge n+1\text{ and }r+j\equiv n\pmod 2.
\end{cases}\]
More explicitly, for fixed \(r\), the sequence has the form
\[
(w_{1,r},\ldots,w_{n,r})
=
(\underbrace{r,\ldots,r}_{n-r\text{ times}},
r-1,r-1,r-3,r-3,r-5,r-5,\ldots),
\]
which is clearly non-increasing and non-negative. 

Now suppose $K_nu=K_nv$. If $n=1$ then clearly $u=v$. For $n\ge 2$, the arguments above force $\sum_{i=1}^ju_i=\sum_{i=1}^jv_i$ whenever there is an $r$ with $w_{j,r}>w_{j+1,r}$. Notice that $w_{n-r,r}-w_{n-r+1,r}=r-(r-1)=1$ for $r=1,\dots,n$ and $w_{n,2}-w_{n+1,2}=1-0=1$, hence $\sum_{i=1}^ju_i=\sum_{i=1}^jv_i$ for $j=1,\dots,n$ and then $u=v$.
\end{proof}

The next lemma essentially relates an arbitrary tournament to the transitive tournament.

\begin{lemma}\label{lem_majorize}
For every skew-symmetric $n \times n$ matrix \(A\) with off-diagonal
entries in \(\{-1,1\}\), and every \(x \in \mathbb{R}^n\), we have\\
\[
(Ax)^*\prec K_n x^*.
\]
\end{lemma} 

\begin{proof} Changing the signs of the coordinates of \(x\) would not
alter \(x^*\). So we may always assume without loss of generality that
all \(x_i \ge 0\) and they are already in a non-increasing order (thus
\(x^*=x\)). If not, we may conjugate \(A\) by the signed permutation
matrix $P$ that gives $x=Px^*$, and we still end up with \(A'=P^T AP\) that is skew-symmetric and has
\(\{-1,1\}\) off-diagonal entries. And 
$$(Ax)^*=(P^TAx)^*=(P^TAPx^*)^*=(A'x^*)^* \prec K_n x^*.$$
The sum of the \(r\) largest absolute values of the coordinates of \(Ax\) is
equal to
\[\max_{\substack{S\subset[n],|S|=r, \varepsilon_i\in\{\pm1\}}}   \sum_{i\in S}\varepsilon_i(Ax)_i.\]
By the skew-symmetry of \(A\) we have
\[\sum_{i \in S} \varepsilon_i (Ax)_i=\sum_{i=1}^n 1_S(i) \varepsilon_i \sum_{j=1}^n a_{ij}x_j = \sum_{i<j} a_{ij}(\varepsilon_i 1_S(i)x_j-\varepsilon_j 1_S(j)x_i).\]
We break the sum on the right hand side into three parts according to whether \(i \in S\) and
\(j \in S\):
\[\sum_{i<j, i \in S, j \in S} a_{ij}(\varepsilon_i 1_S(i)x_j-\varepsilon_j 1_S(j)x_i)+\sum_{i<j, i \in S, j \not\in S} a_{ij}(\varepsilon_i 1_S(i)x_j-\varepsilon_j 1_S(j)x_i)+\sum_{i<j, i \not\in S, j \in S} a_{ij}(\varepsilon_i 1_S(i)x_j-\varepsilon_j 1_S(j)x_i).\]
Note that by exchanging \(i, j\), the third term becomes
\[\sum_{i<j, i \not\in S, j \in S} a_{ij}(-\varepsilon_j x_i)=\sum_{j<i, i \in S, j \not\in S} a_{ji}(-\varepsilon x_j).\]
Therefore in the sum 
\(\sum_{i<j, i \in S, j \not \in S} a_{ij}(\varepsilon_i x_j)+\sum_{i<j, i \not\in S, j \in S} a_{ij}(-\varepsilon_j x_i)\) of the second and third term,
the variable \(x_j\) appears for at most \(|S|=r\) times. Using triangle
inequalities, we get
\[\sum_{i \in S} \varepsilon_i (Ax)_i \le r \sum_{j \not\in S} x_j + \sum_{i<j, i,j \in S} |\varepsilon_i x_j-\varepsilon_j x_i|.\]
Let \(z_i=\varepsilon_i x_i\) then
\(|\varepsilon_i x_j-\varepsilon_j x_i|=|z_i-z_j|\). Sorting \(z_i\) in
descending order as \(\{z'_i\}\) for \(i=1, \cdots, r\). Then
\[\sum_{i<j, i,j \in S} |\varepsilon_i x_j-\varepsilon_j x_i|=\sum_{i<j} (z_i'-z_j')=\sum_{i=1}^r (r-2i+1) z'_i \le \sum_{i=1}^r c_{r,i} x'_i, \]where
\(\{c_{r,i}\}\) is the sequence \(\{|r-2i+1|\}_{i=1,\cdots, r}\) sorted in non-increasing
order, and \(\{x'_i\}_{i=1,\cdots,r}\) are \(\{x_i\}_{i \in S}\) sorted in non-increasing order. Here the last inequality follows from the monotonicity assumption
on \(\{x_i\}\). It is easy to see that \(|c_{r,i}| \le r\), therefore
\[\sum_{i \in S} \varepsilon_i (Ax)_i \le r \sum_{j=1}^{n-r} x_j + \sum_{i=1}^r c_{r,i} x_{n-r+i}.\]
Noting that the sequence of coefficients here is the same as $w_{j,r}$ in Lemma \ref{lem_monotonicity}, and hence the right hand side is exactly \(\sum_{i=1}^r (K_n x)_i\), which completes the proof of the lemma.
\end{proof}

Using these two lemmas repeatedly gives the following corollary.

\begin{corollary} \label{cor_majorize}
Given a sequence \(A_1, \cdots, A_m\) of \(n \times n\) skew-symmetric
matrices with off-diagonal entries from \(\{-1,1\}\), we have
\[(A_1A_2\cdots A_m x)^* \prec K_n^m x^*.\]
Furthermore, if $(A_1A_2\cdots A_m x)^* = K_n^m x^*$, then $(A_jA_{j+1}\cdots A_m x)^* = K_n^{m-j+1} x^*$ for any $j$.
\end{corollary}
\begin{proof} 
We repeatedly apply the two lemmas before. The base case when \(m=1\) follows from
Lemma \ref{lem_majorize}. Suppose it works for \(m=t\), namely
\((A_2 \cdots A_{t+1} x)^* \prec K_n^t x^*\), then for \(m=t+1\), we set
\(u=A_2\cdots A_{t+1}x\) and \(v=K_n^t x^*\), and apply Lemma \ref{lem_majorize}
again for \(m=1\) and Lemma \ref{lem_monotonicity} for \(u^*, v\),  we get
\[(A_1 u)^* \prec K_n u^* = K_n(A_2\cdots A_{t+1}x)^* \prec K_n(K_n^t x^*)=K_n^{t+1} x^*,\]
where if $(A_1u)^*=K_n^{t+1} x^*$ then $K_nu^*=K_nv$ and $u^*=v$.
By induction this proves the corollary.
\end{proof}

We combine the lemmas above to show that $A=T_n$ gives the maximum value of $\Tr[A^{4k}]$ among all skew-symmetric matrices with $\{-1,1\}$-off-diagonal entries.\\

\noindent \textbf{Proof of Theorem \ref{thm_trace}:}\smallskip
\begin{proof} 
Let $e_j$ be the unit vector with the $j$-th coordinate equal to $1$. 
Using Corollary \ref{cor_majorize} for $m=2k$ and $A_i=A$, we have
\[
    (A^{2k} e_j)^* \prec K_n^{2k}(e_j)^* = K_n^{2k} e_1.
\]
Then applying Karamata's inequality (Proposition \ref{prop_karamata}) with $f(t)=t^2$ gives
\[
    (A^{4k})_{jj}=||(A^{2k}e_j)^*||^2_2\le||K_n^{2k}e_1||_2^2=(K_n^{4k})_{11},
\]
where the equality holds if and only if $(A^{2k} e_j)^* = K_n^{2k} e_1$ since $f(t)=t^2$ is strictly convex and strictly increasing when $t\ge0$. Furthermore from the equality condition in Corollary \ref{cor_majorize}, we know that equality holds if and only if
  \[(A^2e_j)^*=K_n^2 e_1.\]
Summing over all
\(j=1, \cdots, n\), we have
\[
    \Tr[A^{4k}] \le n\cdot (K_n^{4k})_{11}.
\]
Next we claim that the RHS is precisely \(\Tr[T_n^{4k}]\). To relate $K_n$ to $T_n$, let us see $K_6$ as an example. 
Let 
\[D=\operatorname{diag}(1,-1,1,-1,1,-1),
\qquad
R=
\begin{bmatrix}
0&0&0&0&0&1\\
0&0&0&0&1&0\\
0&0&0&1&0&0\\
0&0&1&0&0&0\\
0&1&0&0&0&0\\
1&0&0&0&0&0
\end{bmatrix}.\]
Recall that
\[
K_6=  
\begin{bmatrix}  
1&1&1&1&1&0\\
1&1&1&1&0&1\\
1&1&1&0&1&-1\\
1&1&0&1&-1&1\\
1&0&1&-1&1&-1\\
0&1&-1&1&-1&1  
\end{bmatrix}.
\] 
Multiplication by \(R\) reverses the columns:
\[
K_6R=\begin{bmatrix}  
0&1&1&1&1&1\\
1&0&1&1&1&1\\
-1&1&0&1&1&1\\
1&-1&1&0&1&1\\
-1&1&-1&1&0&1\\
1&-1&1&-1&1&0 
\end{bmatrix}.
\]
Multiplication by \(D\) flips the even-numbered rows:
\[DK_6R=
\begin{bmatrix}
0&1&1&1&1&1\\
-1&0&-1&-1&-1&-1\\
-1&1&0&1&1&1\\
-1&1&-1&0&-1&-1\\
-1&1&-1&1&0&1\\
-1&1&-1&1&-1&0
\end{bmatrix}.\]
One can count $1$'s in each column and see that the positions of $1$'s are already similar to $T_6$: just a permutation apart. Namely, let $P$ be the permutation matrix of $\sigma = (1,6,2,5,3,4)$, then $P^TDK_6RP=T_6$.

This process also works for general $K_n$. Let $R$ be the reversal matrix whose anti-diagonal entries are 1 and all other entries are zero, and \[D=\diag(1,-1,1,-1,\ldots), \qquad D_{ii}=(-1)^{i+1}.\] Then define the permutation $\sigma=(1,n,2,n-1,\dots)$, or explicitly,
\[
\sigma(i)=
\begin{cases}
\dfrac{i+1}{2},&i\ \text{odd},\\[2mm]
n+1-\dfrac{i}{2},&i\ \text{even}.
\end{cases}
\]
Let \(P\) be the permutation matrix corresponding to \(\sigma\), we consequently have
\[
P^TDK_nRP=T_n.
\]
Write $S=DP$ we have $T_n^2=-T_nT_n^T=-P^TDK_n^TK_nDP=-S^TK_n^2S$, and hence $T_n^{4k}=S^TK_n^{4k}S$ for all $k$. Notice that $Se_1=e_1$, thus $(T_n^{4k})_{11}=(Se_1)^TK_n^{4k}(Se_1)=(K_n^{4k})_{11}$.

It remains to show that \(T_n^{4k}\) has all the diagonal entries equal, then
this immediately implies $n \cdot (K_n^{4k})_{11}=n \cdot(T_n^{4k})_{11}=\Tr[T_n^{4k}]$.

For this last step, we define the signed cyclic permutation matrix \(C\) with \(C_{j, j+1}=1\) for \(j=1, \cdots, n-1\) and \(C_{n,1}=-1\). It is not hard to check that \(C\) and \(T_n\)
commute. So \(C\) and \(T_n^{4k}\) commute as well. For $i=1,\dots,n-1$, this implies
\[
    (T_n^{4k})_{ii}=e_i^TT_n^{4k}e_i=(Ce_i)^TT_n^{4k}(Ce_i)=(T_n^{4k})_{i+1,i+1},
\]
and hence $(T_n^{4k})_{11}=\dots=(T_n^{4k})_{nn}$ are equal, completing the proof.

Next we study the equality case. It is clear that $A=Q^TT_nQ$ for some signed permutation matrix $Q$ implies $\Tr[A^{4k}]=\Tr[(Q^TT_nQ)^{4k}]=\Tr[Q^TT_n^{4k}Q]=\Tr[T_n^{4k}]$. We shall see that this is the only case where the equality holds. Now suppose $\Tr[A^{4k}]=\Tr[T_n^{4k}]$. Then arguments at the beginning of this proof force $(A^{4k})_{jj}=(K_n^{4k})_{11}$ for all $j$ and hence $(A^{2} e_j)^* = K_n^{2} e_1$. Without loss of generality, we can assume $A_{1j}=1$ for $j=2,\dots,n$ by conjugating $A$ by some signed permutation matrix, i.e.,
\[
    A=\begin{bmatrix}
        0&\vec{1}\\
        -{\vec{1}}^T&B\\
    \end{bmatrix}
\]
for some skew-symmetric matrix $B$ with $\{-1,1\}$-off-diagonal entries. Then 
\[(A^2e_1)^*=(n-1,(B\vec{1})^*)=K_n^{2} e_1=(w_{1,n-1},\dots,w_{n,n-1})=(n-1,n-2,n-2,n-4,n-4\dots).\]
The only possible case to make the first $2$ entries of $(B\vec{1})^*$ to be $n-2$ is to have an all $1$ row and an all $-1$ row. Deleting these two rows and corresponding columns, by induction it is enough to check the case where $B$ is a $1\times 1$ or $2\times 2$ matrix, which is clearly $T_1$ or $\pm T_2$, and we are done.
\end{proof}

\section{Proof of the main theorem}

We begin this section by reducing the cycle counting problem to estimating the trace of matrix powers. We let $J$ be the $n \times n$ all-one matrix.

\begin{proposition}\label{prop_reduction}
    For every \(n\)-vertex tournament $T$, define a skew-symmetric \(n \times n\) matrix \(A\) with off-diagonal entries in \(\{-1, 1\}\) by setting $A_{ij}=1$ if there is an edge from the $i$-th vertex to the $j$-th vertex. Then the number of directed cycles of length $\ell$ in $T$ is equal to
    \[  
    \bigg(\frac{1}{n^{\ell}}\Tr[(A+J)^{\ell}]+O(n^{-1})\bigg) \frac{(\ell-1)!}{2^{\ell}}\binom{n}{\ell}.
    \]
\end{proposition}

\begin{proof}
    Note that $B:=(A+J-I)/2$ is exactly the $\{0,1\}$-adjacency matrix of $T$. Therefore $\Tr(B^\ell)$ counts the number of closed directed walks of length $\ell$ with repetition.
    A genuine \(\ell\)-cycle is counted \(\ell\) times, once for each choice of starting vertex. Closed walks that visit the same vertex more than once contribute only \(O(n^{\ell-1})\), since such a walk uses at most \(\ell-1\) distinct vertices. Therefore, the number of directed cycles of length $\ell$ in $T$ is equal to
    $\Tr[B^\ell]/\ell+O(n^{\ell-1})$. Replacing $B$ with $B+\frac{1}{2}I$ changes the trace by only $O(n^{\ell-1})$. Hence, the number of directed cycles of length $\ell$ in $T$ is equal to \[\Tr[(B+\frac{1}{2}I)^\ell]/\ell+O(n^{\ell-1})= \frac{1}{\ell2^\ell}\Tr[(A+J)^\ell]+O(n^{\ell-1})=\bigg(\frac{1}{n^{\ell}}\Tr[(A+J)^{\ell}]+O(n^{-1})\bigg) \frac{(\ell-1)!}{2^{\ell}}\binom{n}{\ell}.\]
\end{proof}

The next lemma analyzes the spectrum of $T_n$ and the asymptotic behavior of $\Tr[T_n^{4k}]$.

\begin{lemma}\label{lem_Tn}
    $T_n$ has eigenvalues $\lambda_j=i\cot\left(\frac{(2j-1)\pi}{2n}\right)$ for $j=1,\dots,n$, and hence
    \[
    \lim_{n\to\infty} \frac{\Tr[T_n^{4k}]}{n^{4k}} = 2\cdot \sum_{i=1}^{\infty} \left(\frac{2}{(2i-1)\pi}\right)^{4k}.
    \]
\end{lemma}
\begin{proof}
    Let $z_j=e^{\frac{(2j-1)\pi}{n}}$ for $j=1,\dots,n$ be the solutions of $z^n=-1$. Let $v_j=(1,z_j,z_j^2,\dots,z_j^{n-1})^T$, then
    \[
         (Tv_j)_r=\sum_{k=r}^{n-1} z_j^k - \sum_{k=0}^{r-2} z_j^k
                 =\frac{z_j^r-z_j^n-(1-z_j^{r-1})}{1-z_j}
                 =\frac{1+z_j}{1-z_j}z_j^r.
    \]
    Thus, $\lambda_j=\frac{1+z_j}{1-z_j}=i\cot(\frac{(2j-1)\pi}{2n})$ are the eigenvalues of $T_n$ with eigenvectors $v_j$, according to the equality $i\cot(\theta/2)=(1+e^{i\theta})/(1-e^{i\theta})$. Then
    \[  
    \frac{\Tr[T_n^{4k}]}{n^{4k}}=\sum_{j=1}^{n}\left[\frac1n\cot\left(\frac{(2j-1)\pi}{2n}\right)\right]^{4k}.
    \]
    Pairing the terms at \(x\) and \(\pi-x\) gives
    \[
    \frac{\Tr[T_n^{4k}]}{n^{4k}}
    =2\sum_{j=1}^{\lfloor n/2\rfloor}\left[\frac{1}{n}\cot\left(\frac{(2j-1)\pi}{2n}\right)\right]^{4k}.
    \]
    The RHS is bounded by an absolute summable series because for each fixed $j=1,\dots,\floor{n/2}$
    \[
        \frac{1}{n}\cot\left(\frac{(2j-1)\pi}{2n}\right) \le \frac{1}{n}\left(\frac{2n}{(2j-1)\pi}\right) = \frac{2}{(2j-1)\pi}.
    \]
    Therefore we can interchange limit and summation, and get
    \[
     \lim_{n\to\infty} \frac{\Tr[T_n^{4k}]}{n^{4k}} = 2\sum_{j=1}^{\infty}\left[\lim_{n\to\infty}\frac{1}{n}\cot\left(\frac{(2j-1)\pi}{2n}\right)\right]^{4k} = 2\cdot \sum_{i=1}^{\infty} \left(\frac{2}{(2i-1)\pi}\right)^{4k}.
    \]
\end{proof}


\noindent \textbf{Proof of Theorem \ref{thm_main}:}\smallskip

\begin{proof}
\cite{Bartley2018} showed that the number of directed cycles of length \(4k\) in the carousel tournament
is 
\[\left(1+2 \cdot \sum_{i=1}^{\infty} \left(\frac{2}{(2i-1)\pi}\right)^{4k}+o(1)\right)\frac{(4k-1)!}{2^{4k}}\binom{n}{4k},\]
matching the expression in Theorem \ref{thm_main}. It suffices to show that this quantity is also an upper bound.

For every tournament $T$, define $A$ as in Proposition \ref{prop_reduction}. By Proposition \ref{prop_reduction}, it is enough to show that
\[
    \frac{\Tr[(A+J)^{4k}]}{n^{4k}} \le 1+2\cdot \sum_{i=1}^{\infty} \left(\frac{2}{(2i-1)\pi}\right)^{4k}+o(1).
\]
By Theorem \ref{thm_trace} and Theorem \ref{thm_non_negative}, we have
\[  
    \frac{\Tr[(A+J)^{4k}]}{n^{4k}} \le \frac{\Tr[A^{4k}]+\Tr[J^{4k}]}{n^{4k}} \le 1+\frac{\Tr[T_n^{4k}]}{n^{4k}},
\]
and applying Lemma \ref{lem_Tn} completes the proof of the inequality.

Now we give an alternative explanation for why carousel tournaments can attain the upper bound in our proof. Let $C_n$ be the skew-symmetric matrix associated with the carousel tournament when $n$ is odd, then $C_nJ=0$ since each row of $C_n$ has the same numbers of $1$ and $-1$'s. Hence $\Tr[(C_n+J)^{4k}]=\Tr[C_n^{4k}]+\Tr[J^{4k}]$, namely, the equality case of Theorem \ref{thm_non_negative} holds for $C_n$. The equality case of Theorem \ref{thm_trace} also holds because $C_n=Q^TT_nQ$ for signed permutation $Q=DP$, where $D=\diag (\vec{1}_{\frac{n+1}{2}},-\vec{1}_{\frac{n-1}{2}})$, and $P$ is the permutation matrix given by $P_{i,2i-1\bmod n}=1$. In other words, one can apply a series of sign switching to vertices of the transitive tournament and obtain a digraph isomorphic to the carousel tournament.
\end{proof}

We conclude our paper with the two questions. The first one can be viewed as an exact version of Theorem \ref{thm_main}.
\begin{question}
For every fixed \(k\ge 3\) and every sufficiently large odd integer \(n\), does the carousel tournament \(\textup{Car}_n\) always maximize the number of directed \(4k\)-cycles among all \(n\)-vertex tournaments?
\end{question}
In addition to and directed cycles of length not divisible by $4$ and directed paths of arbitrary length, certain other orientations of paths and cycles have been shown to be anti-Sidorenko or Sidorenko. This suggests the following analogous question for carousel tournaments:
\begin{question}
Are there other interesting orientations of paths or cycles for which the carousel tournament asymptotically maximizes or minimizes the number of copies?
\end{question}

\bigskip
\noindent {\bf Acknowledgment. }The authors used ChatGPT 5.6 for the proof of Lemma   \ref{lem_lambda_mu}, minor auxiliary results that lead to the proof of Theorem \ref{thm_trace}, and language polishing. All AI-assisted outputs were independently checked and verified by the authors, who take full responsibility for the content of this work.

\end{document}